\documentclass[letterpaper, twoside, 11pt]{amsart}
\usepackage[T2A]{fontenc}
\usepackage[utf8]{inputenc}	
\usepackage[main=english, russian]{babel}
\usepackage{geometry}
\usepackage{comment}

\usepackage{xcolor}
\definecolor{forestgreen}{HTML}{228B22}
\definecolor{royalblue}{HTML}{4169E1}

\usepackage[unicode, pdftex]{hyperref}
\hypersetup{
	colorlinks   = true, 
	urlcolor     = blue, 
	linkcolor    = royalblue, 
	citecolor    = forestgreen 
}

\usepackage[noabbrev,capitalise,nameinlink]{cleveref}

\usepackage[abbrev,alphabetic,backrefs,initials,nobysame]{amsrefs}

\usepackage{amssymb}
\usepackage{tikz-cd}
\usepackage[all]{xy}
\usepackage{mathrsfs}

\newtheorem{theorem}{Theorem}[section]
\newtheorem{proposition}[theorem]{Proposition}
\newtheorem{lemma}[theorem]{Lemma}
\newtheorem{corollary}[theorem]{Corollary}
\newtheorem{conjecture}[theorem]{Conjecture}

\theoremstyle{definition}
\newtheorem{definition}[theorem]{Definition}
\newtheorem{construction}[theorem]{Construction}
\newtheorem{example}[theorem]{Example}

\theoremstyle{remark}
\newtheorem{remark}[theorem]{Remark}
\newtheorem{warning}[theorem]{Warning}

\newtheorem*{question*}{Questions}

\theoremstyle{plain}
\newcounter{intro}

\newtheorem{intro-conjecture}[intro]{Conjecture}
\newtheorem{intro-corollary}[intro]{Corollary}
\newtheorem{intro-theorem}[intro]{Theorem}
\newtheorem*{intro-corollary*}{Corollary}

\newcommand{\Q}{\mathbb{Q}}
\newcommand{\R}{\mathbb{R}}
\newcommand{\Ko}{\mathbb{C}}
\newcommand{\Zi}{\mathbb{Z}}

\newcommand{\Cc}{\mathfrak{C}}

\newcommand{\X}{\mathcal{X}}
\newcommand{\Xm}{\mathscr{X}}
\newcommand{\B}{\mathcal{B}}
\newcommand{\Bm}{\mathscr{B}}
\newcommand{\Z}{\mathcal{Z}}
\newcommand{\A}{\mathcal{A}}
\newcommand{\Ra}{\mathcal{R}}
\newcommand{\Sy}{\mathcal{S}}
\newcommand{\Cy}{\mathcal{CY}}
\newcommand{\Y}{\mathcal{Y}}
\newcommand{\V}{\mathcal{V}}
\newcommand{\Qp}{\mathcal{Q}}

\newcommand{\Li}{\mathcal{L}}

\newcommand{\Ho}{\mathcal{H}}

\newcommand{\Os}{\mathcal{O}}

\newcommand{\Pp}{\mathbb{P}}
\newcommand{\al}{\alpha}
\newcommand{\Aut}{\textnormal{Aut}}
\newcommand{\gr}{\textnormal{gr}\,}
\newcommand{\D}{\mathbb{D}}

\newcommand{\Ker}{\textnormal{Ker}}
\newcommand{\Burn}{\mathbf{Burn}^m}
\newcommand{\Spec}{\textnormal{Spec}\,}

\newcommand{\CY}{\textnormal{CY}}
\newcommand{\an}{\textnormal{an}}
\newcommand{\prim}{\textnormal{prim}}
\newcommand{\full}{\textnormal{full}}
\newcommand{\tr}{\textnormal{tr}}
\newcommand{\bir}{\textnormal{bir}}

\newcommand{\Iso}{\textnormal{Iso}}

\let\emptyset\varnothing

\title{Towards the Global Torelli Theorem}
\author[D. V. Serebrennikov]{Daniil Serebrennikov}
\address{Department of Mathematics, Johns Hopkins University, Baltimore, MD 21218, USA.}
\email{dserebr1@jhu.edu}

\begin{document}
	\begin{abstract}
		In this paper, we prove finiteness results for K-trivial varieties in all dimensions. In dimension 2, these results are corollaries of the Global Torelli Theorem for polarized K3 surfaces. In particular, we establish a sufficient condition when two fibers of a projective family of K-trivial varieties are isomorphic via the Hodge theory. As applications, we prove that if a projective family of K-trivial varieties has a dense subset of birationally equivalent fibers then all fibers are isomorphic.
	\end{abstract}
	\maketitle
	\tableofcontents
	\section{Introduction}
	In this paper we work with families of polarized Calabi--Yau pairs with at worst klt singularities over the field of complex numbers $\Ko$. For simplicity, we discuss our results for \textit{K-trivial varieties}, that is,  smooth projective varieties $X$ such that the canonical line bundles $\omega_X$ are trivial, i.e. $\omega_X \approx \Os_X$. It is a classical theorem \cite[Historical Sketch, Section 5]{Sha94} due to Riemann, Frobenius, Weierstrass, Poincar\'e, Lefschetz, et al.  that an abelian variety $X$ is determined up to isomorphism by its integral Hodge structure on $H^1(X, \Zi)$, while a K3 surface $X$, i.e. $\omega_X\approx \Os_X$ and $\ h^1(\Os_X) = 0$, is determined by its integral Hodge structure on $H^2(X, \Zi)$ together with the intersection form \cite{PSSh71, BR75, LP81, Fri84} (see also \cite{Mar11, Huy12, Ver13, Loo21} for generalizations to hyperk\"ahler manifolds). By contrast, for K-trivial threefolds $X$ with $h^1(\Os_X) = 0$, even the polarized integral Hodge structure on $H^3(X, \Zi)$ does not determine the variety up to isomorphism (resp. up to birational equivalence) by \cite{Sze00} (resp. by \cite{OR18, BCP20, Ram24}). Although uniqueness is proven for some classes of varieties \cite{Voi99, Voi22}, we focus on finiteness as a natural weakening of uniqueness.
	
	\begin{question*}
		 Is an isomorphism/birational type of a K-trivial variety $X$ determined by its polarized periods {up to finitely many possibilities}? Does $X$ admit at most finitely many polarizations?
	\end{question*}
	
	Note that even the original Torelli theorem \cite{Tor13} with \cite{NN81} implies that there are at most finitely many isomorphism types of smooth projective curves whose periods coincide. To formalize these questions, let us show what finiteness results follow from the Global Torelli Theorem for polarized K3 surfaces.
	A \textit{polarized K3 surface} is a pair $(X, A)$ consisting of a K3 surface $X$ and an (primitive) ample line bundle $A$ on $X$. Its degree is the self-intersection number $A^2$.  According to the Global Torelli Theorem (see \cite[Chapter 7]{Huy16} and references therein), the period map from the course moduli space $M_{2, N}^\text{K3}$ for polarized K3 surfaces of degree $N$ is an open embedding of quasi-projective schemes:
	\begin{equation*}
		\Phi^\text{K3}: M_{2, N}^{\textnormal{K3}} \hookrightarrow \Gamma^\text{K3} \backslash \mathbb{D}^\text{K3}
	\end{equation*}
	 where $\mathbb{D}^\text{K3}$ is a period domain for polarized K3 surfaces, and $\Gamma^\text{K3}$ is an arithmetic group preserving the polarization. Next, for every smooth family $(\X/S, \mathcal{A})$ of polarized K3 surfaces of degree $N$, there is a period map $\Phi: S\to \Gamma^\text{K3} \backslash \mathbb{D}^\text{K3}$ which factors through $\Phi^\text{K3}$.  Then we obtain the following.
	\begin{enumerate}
		\item 
			 For any point $\star\in \Gamma^\text{K3} \backslash \mathbb{D}^\text{K3}$ we have $(\X_s, \A_s)\approx (\X_t, \A_t)$ for all (closed) points $s,t\in \Phi^{-1}(\star)$.
		\item 
			The global period map $\Phi^\text{K3}: M_{2, N}^\text{K3}\hookrightarrow \Gamma^\text{K3} \backslash \mathbb{D}^\text{K3}$ is quasi-finite, i.e. all fibers are finite sets.
		\item 
			For every K3 surface $X$ the set $\{ \ [(X', A')]\in M_{2, N}^\text{K3}: X'\approx X \ \}$ is finite by \cite{Ste85}.
			
	\end{enumerate}
	
	Note that property (3) is the most non-trivial corollary, whose generalization to K-trivial threefolds $X$ with $h^1(\Os_X) = h^2(\Os_X) = 0$ was treated in \cite{Sze99}. Also, it is known for hyperk\"ahler manifolds \cite{Huy18}.  Although the existence of the coarse moduli space for smooth polarized K-trivial varieties has been known for a long time \cite{Vieh95},  a natural generalization even of property~(1) has not been explicitly addressed yet.
		\begin{intro-theorem}
			\label{th-intro: connected_components}
			Let $f: \X\to S$ be a smooth projective morphism between regular varieties, and every fiber $\X_s$ is K-trivial. Let $\Phi: S\to \Gamma \backslash \mathbb{D}$ be any geometric period map associated with a variation of polarized integral Hodge structures of weight $d = \dim \X/S$. Then for every point $\star\in \Gamma \backslash\mathbb{D}$ and every connected component $P$ of $(\Phi)^{-1}(\star)$ we have $\X_s \approx \X_t$ for all (closed) points $s, t\in P$.  
		\end{intro-theorem}

	Let us make several remarks regarding this theorem. First, there are many \textit{geometric period maps} (see \cref{constr: period_maps_smooth}) associated to a smooth projective family $\X/S$ depending on a choice of an ample  line bundle $\mathcal{A}$ on $\X/S$, and a choice of a sublattice of $H^d(\X_s, \Zi)_\text{t.f.}$, e.g. the full lattice $H^d(\X_s, \Zi)_\text{t.f.}$, the primitive lattice $H^d_\prim(\X_s, \Zi)$, and the transcendental one --- the lattice induced by the smallest rational Hodge sub-structure including $H^{d,0}(\X_s)$. Arguably, the latter is the most natural choice because it allows one to define period maps for lc-trivial fibrations (see \cref{const: main}). Second, the proof of \cref{th-intro: connected_components} uses only standard arguments from the deformation theory, in particular, infinitesimal Torelli theorem. Meanwhile, its generalization (\cref{th: main}) to lc-trivial fibrations involves \cite{Amb05} and recent results on a singular Beauville--Bogomolov decomposition \cite{BGL22,MW25}. Third, \cite[Section 7.3]{BFMT25} shows that property (2) holds after an appropriate stratification of the coarse moduli space $M_{d, N}$, and the natural maps to $M_{d, N}$ are locally constant for isotrivial families. So, \cref{th-intro: connected_components} is presumably known to experts. However, an unconditional generalization of property~(3) has not been known.

		\begin{intro-theorem}
			\label{th-intro: finiteness}
			Let $M_{d, N}$ be the coarse moduli space for polarized smooth $K$-trivial varieties over $\Ko$ of degree~$N$. Then for every K-trivial variety $X$ the set
			\begin{equation*}
				\{ \ [(X', A')]\in M_{d, N}: X'\sim_\textnormal{bir} X \ \}
			\end{equation*}
			is finite. In addition, there is stratification of $M_{d, N}$ by (locally closed) quasi-projective schemes $U_i$, on each of which there is a quasi-finite period map.
		\end{intro-theorem}
			This theorem heavily relies on our previous work \cite{Ser26}, and widely generalizes \cite{NN81, Ste85, Sze99, Huy18}. In addition, neither the present paper nor \cite{Ser26} uses main results  from \cite{BFMT25} on the b-semiampleness conjecture and on the compactifications for $M_{d, N}$. Nevertheless, we use algebraicity of period maps \cite{BBT23} throughout the paper. For applications, we prove
		
		\begin{intro-theorem}
			\label{th-intro: rigidity}
			Let $f: \X\to S$ be a smooth projective morphism between regular varieties, and every fiber $\X_s$ is K-trivial. Suppose that there is a Zariski dense subset $U\subseteq S$ of closed points such that $\X_s\sim_\text{bir} \X_t$ for all $s,t \in U$. Then $\X_s \approx \X_t$ for all closed points $s,t \in S$.
		\end{intro-theorem}
		
		Moreover, we prove analogues of \cref{th-intro: connected_components}, \cref{th-intro: finiteness}, \cref{th-intro: rigidity} (cf. \cref{th: main}, \cref{th: moduli_finite}, \cref{th: 0-pair_iso}) for families of projective klt Calabi--Yau pairs $(X, B)$, that is, $K_X+B\equiv 0$, and $B$ is an effective $\Q$-divisor. 
		Finally, we establish a birational specialization theorem whose proof proceeds along the same lines as in \cite{CLKT26}. We say a smooth proper variety $X$ is \textit{K-torsion} if $\omega_X$ is torsion, equivalently we have  $K_X\sim_\Q 0$ for a canonical divisor. We write $(X, 0)\sim_\text{bir} (X', 0)$ if $X\sim_\text{bir} X'$, and for any model dominating both $X, X'$ the pullbacks of $K_X$ and $K_{X'}$ coincide.
		
		\begin{intro-theorem}
			\label{th-intro: specialization}
			Let $f: \X\to S$ and $f': \X'\to S$ be  smooth proper morphisms between regular varieties  such that $K_\X \sim_{\Q, S} 0, \ K_{\X'} \sim_{\Q, S} 0$. Suppose that $(\X_\eta, 0)\sim_\textnormal{bir} (\X'_\eta, 0)$ for the generic point $\eta\in S$. Then $(\X_s, 0)\sim_\textnormal{bir} (\X'_s, 0)$ for all closed points $s\in S$.
		\end{intro-theorem}
		In an earlier draft, the analogue of \cref{th-intro: specialization} for klt pairs was used in the proofs of the main results, but it is no longer needed in the present version. We replace it by the Matsusaka--Mumford \cite{MM64} specialization argument which was generalized by Koll\'ar \cite{Kol85, Kol23}. Nevertheless, since the theorem may be of independent interest and useful in other contexts, we have retained it.
		\begin{remark}
			In \cite{LL26} authors investigated when an isotrivial over open dense subset $U\subseteq S$ smooth projective family $\X/S$ of good minimal models is isotrivial for K\"ahler morphisms over analytic polydiscs $S$. This could be used for an alternative proof of \cref{th-intro: connected_components}. 
		\end{remark}
		\begin{remark}
			In \cite[Section 9]{Shok13}, its author proved an analogue for property (3) in all dimensions by different methods based on the standard conjectures of the Minimal Model Program.
		\end{remark}
		\addtocontents{toc}{\protect\setcounter{tocdepth}{1}}
		\subsection*{Acknowledgments}
		The author is grateful to his advisor, Professor Shokurov, for his support and encouragement as well for several helpful discussions.
	\section{Preliminaries}
	\subsection{Notation and conventions.}
	Throughout, we work in the category of Noetherian schemes over $\mathbb{C}$ and follow the standard terminology \cite{IS05} unless stated otherwise. By a variety we mean an integral separated scheme of finite type over a field. On a variety, a prime divisor is understood in the Weil sense, i.e. an integral closed subscheme of codimension~$1$. In every part of the paper, for an $\mathbb{R}$-divisor $D$ the sum $D=\sum_i d_i D_i$ assumes that all $D_i$ are distinct prime divisors, and $d_i\in\mathbb{R}$.
	
	A \textit{pair} $(X,B)$ consists of a variety $X$ and an $\mathbb{R}$-divisor $B$ on $X$. We say that $(X,B)$ is a \textit{log pair} if $X$ is normal, and $K_X+B$ is an $\mathbb{R}$-Cartier $\R$-divisor. Suppose $B=\sum_i b_i B_i$, and $\Gamma\subseteq\mathbb{R}$ is a set. If $b_i\in\Gamma$ for all~$i$ then we write $B\in\Gamma$. Denote by $[0,1]$ the unit segment of real numbers.  The $\mathbb{R}$-divisor $B$ is called a \textit{boundary} if $B\in [0,1]$.	The pairs $(X,B)$ and $(X',B')$ are said to be \textit{log isomorphic} if there exists an isomorphism $\varphi:X\to X'$ such that $\varphi_*(B)=B'$; we then write $(X,B)\approx (X',B')$. We use the notation $\cong$ instead of $\approx$ if an isomorphism is natural. We write $g: (Y, D) \to (X, B)$ if $g: Y\to X$ is a morphism such that $K_Y + D = g^*(K_X + B)$.
	
	\begin{definition}
		\label{def: sing}
		Let $(X,B)$ be a log pair with $B=\sum b_i B_i$. We say that $(X,B)$ has (at worst) the following singularities when the corresponding inequalities hold:
		\[
		\begin{gathered}
			\begin{cases}
				\text{terminal (trm)} \\
				\text{canonical} \\
				\text{Kawamata log terminal (klt)}\\
				\text{log canonical (lc)}
			\end{cases}
			\Longleftrightarrow
			\text{dis}(X,B)
			\begin{cases}
				>0;\\
				\ge0;\\
				>-1 \ \text{and}\  b_i<1 \ \forall i;\\
				\ge -1 \ \text{and}\  b_i\le 1 \ \forall i.
			\end{cases}
		\end{gathered}
		\]
		Here, the \textit{discrepancy} of $(X, B)$ is given by $$\text{dis}(X, B) = \underset{E}{\inf}\{a(E; X, B): E \text{ is an exceptional prime divisor over } X\},$$ that is, $E$ runs through all the prime exceptional divisors of all birational morphisms $g: Y\to X$ from normal varieties $Y$, and the number $a(E; X, B)\in \mathbb{R}$ is defined by the formula $$K_Y = g^*(K_X + B) + \sum_E a(E; X, B)\, E.$$
		Similarly, $(X, B)$ is $\epsilon$\textit{-klt} (for some $\epsilon > 0$) if $\text{dis}(X, B) > -1 + \epsilon$ and $b_i < 1$ for all $i$.
	\end{definition}
	
	\begin{definition}
		\label{def: 0-pair}
		A log pair $(X,B)$ is called a \textit{weakly log canonical model (wlc model)} if the following assumptions hold:
		\begin{itemize}
			\item $X$ is a proper variety, and $B$ is a boundary.
			\item $(X,B)$ has log canonical singularities.
			\item $K_X+B$ is nef.
		\end{itemize}
		If $(X,B)$ has (at worst) klt (resp. terminal) singularities, then $(X,B)$ is called a \textit{wlc klt model} (resp. a \textit{wlc trm model}). If, in addition, $K_X+B\equiv 0$, then $(X,B)$ is called a \textit{0-pair}\footnote{In the literature, 0-pairs are also called Calabi–-Yau pairs.}.
	\end{definition}
	
	We shall write $X\sim_\text{bir} X'$ if the varieties $X, X'$ are birational. In the following definition we introduce an appropriate notion of birational equivalence for log pairs. 
	\begin{definition}
		\label{def: 0-class}
		We say that a log pair $(X_\alpha,B_\alpha)$ is \textit{crepant birationally equivalent} or \textit{0-equivalent} to $(X,B)$ if the varieties  $X_\alpha$ and $X$ are birationally equivalent, and for each common log resolution $Y$ of these pairs there is an $\R$-divisor $D$ on $Y$ such that the following equalities hold:
		\[
		\begin{gathered}
			K_Y+D=g^*(K_X+B)=g_\alpha^*(K_{X_\alpha}+B_\alpha),\\
			B=g_*D,\quad B_\alpha=(g_\alpha)_*D,
		\end{gathered}
		\]
		where $g:Y\to X$ and $g_\alpha:Y\to X_\alpha$ are the corresponding log resolutions. The class of all log pairs crepant birationally equivalent to $(X,B)$ is called the \textit{0-class} of the log pair $(X,B)$. In this situation, we write $(X, B)\sim_\bir (X_\al, B_\al)$
	\end{definition}
	\begin{remark}
		If $(X,B)$ is a wlc klt model (resp. 0-pair), then each log pair $(X_\alpha,B_\alpha)$ crepant birationally equivalent to $(X, B)$ with effective $\mathbb{R}$-divisor $B_\alpha$ is also a wlc klt model (resp. 0-pair).
	\end{remark}
	\begin{warning}
		The equivalence $X\sim_\text{bir} X'$ does not imply $(X, 0) \sim_\text{bir} (X', 0)$, e.g. a projective plane $X = \Pp^2_\Ko$ and its blow-up $X' = \text{Bl}_p \Pp^2_\Ko$ at some point $p\in \Pp^2_\Ko$.
.	\end{warning}

	\begin{example}
		\label{ex: flops_0-equivalent}
		Let $(X, B), (X^\prime, B^\prime)$ be projective klt wlc models, $t: (X_\text{trm}, B_\text{trm}) \to (X, B)$ and $t^\prime: (X_\text{trm}^\prime, B_\text{trm}^\prime)\to (X^\prime, B^\prime)$ be their $\Q$-factorial terminalizations respectively \cite[Corollary 1.4.3]{BCHM10}. Suppose there is a birational map $\varphi: X \dashrightarrow X^\prime$ such that $\varphi_*B_\text{trm} = B^\prime_\text{trm}$. Then $(X_\text{trm}, B_\text{trm})$ and $(X^\prime_\text{trm}, B^\prime_\text{trm})$ are 0-equivalent according to \cite[Theorem 1]{Kaw08}. Hence $(X, B), (X^\prime, B^\prime)$ are 0-equivalent as well.
	\end{example}
	
	Let $f:\X\to S$ be a morphism of schemes and $s\in S$ a point. We denote the scheme-theoretic fiber by $\X_s=\X\times_S s$. If $\A$ is an invertible sheaf on $\X$, then its restriction to the fiber $\X_s$ is denoted by $\A_s=\A|_{\X_s}$. When no ambiguity arises, we denote a morphism $f: \X \to S$ simply by $\X/S$. If $\A$ is (very) ample over $S$, we  say that $\A$ is an (very) ample invertible sheaf on $\X/S$.
	
	\begin{definition}
		Suppose $f: \X \to S$ is a proper flat morphism between varieties, and the base $S$ is regular. We say that $\X/S$ is a \textit{family of varieties} if every fiber $\X_s$ is a variety. If, in addition, the morphism $f$ is projective (resp. smooth), then $\X/S$ is called a projective (resp. smooth) family of varieties.
	\end{definition}
	
	For a family of varieties $\X/S$ and a proper variety $X$ we introduce the following sets:
	\begin{equation*}
		\begin{gathered}
			\Iso_X(\X/S) = \{s\in S: \X_s\times_\Ko \overline{\kappa(s)}\approx X\times_\Ko\overline{\kappa(s)}\},\\ \text{Bir}_X(\X/S) = \{s\in S: \X_s\times_\Ko\overline{\kappa(s)}\sim_\bir X\times_\Ko\overline{\kappa(s)}\},
		\end{gathered}
	\end{equation*}
	where $\kappa(s)$ is the residue field of the point $s\in S$, and $\overline{\kappa(s)}$ is its algebraic closure.
	\begin{definition}
		\label{def: isotriviality}
		Let $\X/S$ be a family of varieties. We say $\X/S$ is \textit{isotrivial} (resp. \textit{birationally isotrivial}) if $\Iso_X(\X/S) = S$ (resp. $\text{Bir}_X(\X/S) = S$) for some proper variety $X$. Also, $\X/S$ is \textit{potentially isotrivial} (resp. \textit{potentially birationally isotrivial}) if for some proper variety $X$ the set $\Iso_X(\X/S)$ (resp. $\text{Bir}_X(\X/S)$) includes a dense subset of closed points in $S$. 
	\end{definition}
	\begin{warning}
		The notion of \textit{isotriviality} varies considerably across the literature (compare with \cite[Proposition 2.610]{Ser06} and \cite[Theorem 1.1]{BBG16}). 
	\end{warning}
	
	Consider a family of varieties $\X/S$ and an $\mathbb{R}$-divisor $\B=\sum b_i\B_i$ on $\X$. In general, the restriction of $\B$ to a fiber $\X_s$ is not defined, and even when it is, the definition is not straightforward. For our purposes it is convenient to work with \textit{elementary families} on which $\B_s$ is defined.
	
	\begin{definition}
		\label{def: elementary_family}
		Let $(\X,\B)$ be a pair and $f:\X\to S$ be a family of varieties. Suppose $\B=\sum_{i} b_i\B_i$ is a decomposition of the $\mathbb{R}$-divisor into distinct prime divisors. We say that $(\X/S,\B)$ is an \textit{elementary family} if for every point $s\in S$ the following hold:
		\begin{enumerate}
			\setcounter{enumi}{-1}
			\item The fiber $\X_s$ is a normal variety.
			\item The restriction $f|_{\B_i}:\B_i\to S$ is flat for all $i$.
			\item The closed subscheme $(\B_i)_s=\B_i\times_S s$ is a prime divisor on $\X_s$ for all $i$.
			\item The prime divisors $(\B_i)_s$ and $(\B_j)_s$ are distinct whenever $i\ne j$.
		\end{enumerate}
		In this situation, we put $\B_s=\sum_i b_i(\B_i)_s$. If, in addition, the morphism $f$ is projective, then $(\X/S,\B)$ is called a \textit{projective elementary family}.
	\end{definition}
	Now, we can generalize \cref{def: isotriviality} to log pairs. For an elementary family $(\X/S, \B)$ and a proper log pair $(X, B)$ we set:
	\begin{equation*}
		\begin{gathered}
			\Iso_{(X, B)}(\X/S, B) = \{s\in S: (\X_s, \B_s)\times_\Ko \overline{\kappa(s)}\approx (X, B)\times_\Ko\overline{\kappa(s)}\},\\ \text{Bir}_{(X, B)}(\X/S, \B) = \{s\in S: (\X_s, \B_s)\times_\Ko\overline{\kappa(s)}\sim_\bir (X, B)\times_\Ko\overline{\kappa(s)}\}.
		\end{gathered}
	\end{equation*}
	
	\begin{definition}
		\label{def: isotriviality_pairs}
		Let $(\X/S, \B)$ be an elementary family. We say $(\X/S, \B)$ is \textit{isotrivial} (resp. \textit{birationally isotrivial}) if $\Iso_{(X, B)}(\X/S, \B) = S$ (resp. $\text{Bir}_{(X, B)}(\X/S, \B) = S$) for some proper log pair $(X, B)$. Also, $(\X/S, \B)$ is \textit{potentially isotrivial} (resp. \textit{potentially birationally isotrivial}) if  for some proper log pair $(X, B)$ the set $\Iso_{(X, B)}(\X/S, \B)$ (resp. $\text{Bir}_{(X, B)}(\X/S, \B)$) includes a dense subset of closed points in $S$. 
	\end{definition}
	For applications, we will restrict ourselves to a subclass of elementary families for which all fibers have prescribed singularities, and their log canonical divisors are compatible with the log canonical divisor of the total family. More precisely we have
	
	\begin{definition}
		\label{def: elementary_family(very)}
		Suppose $(\X/S, \B)$ is an elementary family. We say $(\X/S, \B)$ is an \textit{elementary log family} (resp. \textit{elementary klt family}) if $(\X, \B)$ is a log pair (resp. klt log pair), $(\X_s, \B_s)$ is a log pair (resp. klt log pair), and $(K_\X + \B)|_{\X_s} = K_{\X_s} + \B_s$ for all points $s\in S$. An \textit{elementary lc family} is defined similarly.
	\end{definition}
	\begin{remark}
		An elementary klt/lc family is locally stable in the sense of Koll\'ar \cite{Kol23}.
	\end{remark}
	\begin{definition}
		Let $(\X/S, \B)$ be an elementary family, $\B = \sum_{i\in I} b_i \B_i$. We say that $(\X/S, \B)$ \textit{is log smooth over} $S$ if $\X/S$ is smooth, $\B$ has relative snc (simple normal crossing) support, and $\bigcap_{j\in J}\B_{j}/S$ is smooth for all $J\subseteq I$. Here $\bigcap_{j\in J}\B_{j}$ is a scheme with the structure of a reduced closed subscheme.
	\end{definition}
	\subsection{Burnside groups}
	In \cite{CLKT26}, the authors introduced Burnside rings for logarithmic forms. It is mentioned in \textit{loc. cit.} that these notions could be adapted to pluriforms. To make our exposition more self-contained, we spell out these adaptations, state main results, and sketch some proofs.
	
	From now and on, we fix a field $k$ of characteristic $0$. Let $X$ be a proper integral scheme over $k$, $\dim X = d$. By $\Omega_{k(X)}^d$ we denote the $k(X)$-vector space of rational differential $d$-forms on $X$. For every $m\in \mathbb{Z}_{>0}$ the elements of $(\Omega_{k(X)}^{d})^{\otimes m}$ are called \textit{(rational) volume $m$-pluriforms}. For $\omega\in (\Omega^d_{k(X)})^{\otimes m}$ by $\text{div}(\omega) = D_{\ge 0} - D_{\le 0}$ we denote the divisor of zeros $D_{\ge 0}$ and poles $D_{\le 0}$  where the latter is called \textit{the polar divisor for $\omega$}.  We say that $\omega\in (\Omega_{k(X)}^{d})^{\otimes m}$ is  a \textit{logarithmic (rational) volume $m$-pluriform on $X$} if the form $g^*\omega$ has poles of order at most $m$ for all proper birational morphisms $g: X'\to X$ such that the polar divisor for $g^*\omega$ has simple normal crossing\footnote{By \cite[Lemma 2.4]{CLKT26}, it is sufficient to check the condition on poles for one proper birational morphism $g\colon X'\to X$ such that the polar divisor for $g^*\omega$ has simple normal crossing.}.

	\begin{definition}
		Let $S$ be a scheme over $k$, $d\in \mathbb{Z}_{\ge 0}, m\in \mathbb{Z}_{>0}$. By $\Burn_d(S/k)$ we denote the free abelian group generated by isomorphism classes of pairs $(X/S, \omega)$ such that
		\begin{itemize}
			\item
			$X$ is a smooth proper variety over $k$ of dimension $d$ with a morphism $X\to S$,
			\item 
			$\omega$ is a logarithmic (rational) volume $m$-pluriform on $X$,
		\end{itemize}
		subject to the smallest equivalence relation such that $(X/S, \omega) \sim (X'/S, \omega')$ if and only if there is a proper birational morphism  $g: X'\to X$ over $S$ with $g^*\omega = \omega'$ (compare with \cref{def: 0-class}). For every morphism of schemes $h: S\to T$ over $k$ there is a natural morphism of groups $$h_*: \Burn_d(S/k) \to \Burn_d(T/k).$$
		
		We write $[X/S, \omega]\in \Burn_d(S/k)$ for the class of a pair $(X/S, \omega)$.
		The direct sum $$\Burn(S/k) = \bigoplus_{d\ge 0} \Burn_d(S/k)$$ view as a graded abelian group is called a \textit{Burnside group of index $m$}. When $S =\text{Spec}\, k \to \text{Spec}\, k$ is the identity morphism, we write $\Burn(S/k):= \Burn (k)$.
	\end{definition}
	
	\begin{proposition}
		\label{prop: ext_fields}
		Let $k'/k$ be an extension. Then there is a natural morphism of groups $$\Burn(S/k) \to \Burn(S'/k'),$$
		where $S' = S\times_k k'$
	\end{proposition}
	\begin{proof}
		Let $(X, \omega)$ be a smooth proper integral scheme over $k$ equipped with a logarithmic volume $m$-pluriform. Let $X' = X\times_k k'$ be its base change to $k'$, and $\omega'$ be the volume $m$-pluriform obtained by base change. Then for every irreducible component $Y$ of $X'$ we have that $(Y, \omega'|_Y)$ is a logarithmic volume $m$-pluriform. Hence, we define $(X, \omega) \mapsto \sum_{Y\subset X'}(Y, \omega'|_Y)$ where the finite sum runs over all irreducible components of $X'$.
	\end{proof}
	
	Let $C$ be a regular integral curve over $k$, let $o\in C$ be a closed point such that $\kappa(o) = k$. By $\eta$ we denote the generic point of $C$. Put $K = k(C)$, $R = \Os_{o, C}$, and fix a uniformizer $t\in R$. Consider a smooth proper variety $X$ over $K$ with a logarithmic volume $m$-pluriform $\omega_\eta$ on $X$, and $\dim X = d$. Let $\mathscr{X}/C$ be a regular flat proper model for $X$, $\Delta = (\Xm_o)_\text{red}$ be its reduced special fiber, and $\mathscr{B}$ be a $\Q$-divisor on $\Xm$ that has relative snc, and $m\mathscr{B}$ is Cartier. We assume that the $\Q$-divisor $\Delta + \Bm$ has snc, and $(\Xm, \Delta + \Bm)$ is log canonical.
	\begin{definition}
		A (rational) relative volume $m$-pluriform on $\Xm/C$ is \textit{logarithmic} with respect to $\Delta + \Bm$ if it is (locally) the image of a logarithmic volume $m$-pluriform $\tilde{\omega}\in (\Omega^d_{\Xm/k})^{\otimes m}$ with poles contained in $m(\Delta + \Bm)$ under the natural morphism $(\Omega^d_{\Xm/k})^{\otimes m} \to (\Omega^d_{\Xm/C})^{\otimes m}$.
	\end{definition}
	
	Let $\Delta = \sum_{i\in I} \Delta_i$ and $\Bm = \sum_{j\in J} b_j\Bm_j$. For any $I'\subseteq I$ (resp. $J'\subseteq J$) we put $\Delta_{I'} = \bigcap_{i\in I'} \Delta_i$ (resp. $\Bm_{J'} = \bigcap_{j\in J'} \Bm_j$). Consider a relative logarithmic  with respect to $\Delta + \Bm$ volume $m$-pluriform $\omega$ on $\Xm/C$, and a volume $m$-pluriform $\omega$ on $\Xm$ defined locally by $\omega'  = \tilde{\omega}\wedge dt/t^m$ where $\tilde{\omega}$ is any local lift of $\omega$ to $\Xm/k$. Then there are subsets $I_0\subseteq I, J_0\subseteq J$ such that the polar divisor for $\omega'$ is given by $m(\sum_{i\in I_0} \Delta_i + \sum_{j\in J_0}b_jB_j)$.
	
	Now we define an element of $\Burn_d (\Xm_o/k)$ by the following formula:
	\begin{equation*}
		\rho_t(\Xm, \omega) = \sum_{\substack{\emptyset\neq I'\subseteq I_0,\\ J'\subseteq J_0}} (-1)^{|I'| + |J'|-1 }\rho_{\Delta_{I'}\cap \Bm_{J'}}(\Xm, \omega)
	\end{equation*}
	where $\rho_{\Delta_{I'}\cap \Bm_{J'}}$ is the residue map locally given by $\rho_{\Delta_{I'}\cap \Bm_{J'}}(\tilde{\omega}\wedge dt/t^m) = \tilde{\omega}|_{\Delta_{I'}\cap \Bm_{J'}}$.
	
	\begin{proposition}[{\cite[Proposition 8.10]{CLKT26}}]
		\label{prop: independence_model}
		Let $\mathscr{Z}\subseteq \Xm$ be an irreducible regular closed subscheme of $\Xm$ which is transverse to $\Delta + \Bm$. Let $g: \Xm'\to \Xm$ be a blow-up of $\Xm$ along $\mathscr{Z}$. The form $g^*\omega$ is a relative logarithmic with respect to $g^*(\Delta + \Bm)$ volume $m$-pluriform on $\Xm'/C$. Moreover, we have
		\begin{equation*}
			g_*\rho_t(\Xm', g^*\omega) = \rho_t (\Xm, \omega)\in \Burn_d (\Xm_o/k).
		\end{equation*}
 	\end{proposition}
 	
 	Recall that $\omega_\eta$ is a logarithmic volume $m$-pluriform on $X$ over $K$. As in \cite[Section 8.13]{CLKT26}, after a finite base change $C'\to C$ there is an integer $r = r(\omega_\eta)$ such that $t^{r}\omega_\eta$ extends to a relative logarithmic volume $m$-pluriform $\omega$ on $\Xm/C$. From \cref{prop: independence_model} and the weak factorization theorem \cite{AKMW02} it follows that the element $h_*\rho_t(\Xm, \omega) \in \Burn_d(k)$ is independent of the choice of a model $\Xm/C$ where $h: \Xm_o\to k$ is the structure morphism. Hence, we obtain
	
	\begin{theorem}[{\cite[Theorem 8.14]{CLKT26}}]
		\label{th: specialization}
		Let $K$ be a field of fractions  of a discrete valuation ring $R$ with residue field $k$. Then for any uniformizer $t\in R$ there is a morphism of groups $$\rho_t: \Burn(K)\to \Burn(k).$$
	\end{theorem}
	
	\begin{theorem}
		\label{th: bir_spec}
		Let $(\X/S, \B)$, and $(\X'/S, \B')$ be elementary lc families such that both $(\X, \B), (\X', \B')$ are log smooth over $S$, and $K_\X + \B \sim_{\Q, S} 0, \ K_{\X'} + \B'\sim_{\Q, S} 0$. Suppose that $(\X_\eta, \B_\eta)\sim_\text{bir}(\X_\eta', \B'_\eta)$ for the generic point $\eta\in S$. Then $(\X_s, \B_s)\sim_\text{bir}(\X_s', \B'_s)$ for all closed points $s\in S$.
	\end{theorem}
	\begin{proof}
		Let $m\in \mathbb{Z}_{> 0}$ be the minimal integer such that $m\B$ is Cartier and $mK_\X \sim_S -m\B$. Then  there is a relative logarithmic volume $m$-pluriform $\omega$ on $\X$ (resp. $\omega'$ on $\X'$) such that $\text{div}(\omega) = -m\B$ (resp. $\text{div}(\omega') = -m\B'$). By the hypotheses, we obtain
		$[\X_\eta, \omega_\eta]= [\X'_\eta, \omega'_\eta] \in \Burn(K)$ where $K = \kappa(\eta)$. Fix a closed point $o\in S$, an affine open neighborhood $U = \text{Spec}\, A$,  a general regular prime divisor $  \text{Spec}\, A/\mathfrak{p}\subseteq U$ passing through $o$,  and we choose a uniformizer $t \in R:=  A_\mathfrak{p}$. Let $\widehat{K_\mathfrak{p}}$ be the completion of $K$ at $\mathfrak{p}$, that is, a local field with residue field $\widehat{\kappa(\mathfrak{p})}$ isomorphic to $\kappa(\mathfrak{p})((t))$. Let $\iota_s: K\hookrightarrow \widehat{K_\mathfrak{p}}$ be the canonical inclusion, and $(\iota_{\mathfrak{p}})_*: \Burn(K) \to \Burn(\widehat{K_\mathfrak{p}})$ the induced morphism (\cref{prop: ext_fields}). According to \cref{th: specialization} there is a specialization morphism $\rho_t: \Burn(\widehat{K_\mathfrak{p}}) \to \Burn(\widehat{\kappa(\mathfrak{p})})$. From the definition of $\rho_t$ it follows that:
		\begin{equation*}
			[\X_\mathfrak{p}, \omega_\mathfrak{p}] = \rho_t\circ(\iota_{\mathfrak{p}})_* [\X_\eta, \omega_\eta] = \rho_t\circ(\iota_{\mathfrak{p}})_*[\X'_\eta, \omega'_\eta] = [\X'_\mathfrak{p}, \omega'_\mathfrak{p}].
		\end{equation*}
		By induction on dimension of the regular base $S$, we obtain the equality $[\X_s, \omega_s] = [\X'_s, \omega'_s]$ in $\Burn(\widehat{\kappa(s)})$ where $\widehat{\kappa(s)} \approx \Ko((x))$ for some formal local coordinate $x$. This implies that $[\X_s, \omega_s] = [\X'_s, \omega'_s]$ in $\Burn(\Ko)$ by the standard descent argument. From the constructions of $\omega, \omega'$ it follows that $(\X_s, \B_s)\sim_\text{bir}(\X_s', \B'_s)$. Since the point $o\in S$ was chosen arbitrarily, this completes the proof. 
	\end{proof}
	
	Now, let us compare \cref{th: bir_spec} with the Matsusaka--Mumford specialization results \cite{MM64}. The former is merely an existence result for proper varieties while the latter gives a sufficient condition when the specialization of a birational map from the generic fiber is an isomorphism in the projective category (see \cite[Proposition 3.1.2]{Kol85}). For our applications, the second approach turned out to work better. We state and prove a simplified modern version \cite[Theorem 11.40]{Kol23}.
	
	\begin{definition}
		Let $(X, B)$ be a lc pair, and $W\subseteq X$ be an irreducible closed subset.  The \textit{minimal log discrepancy} of $W$ is
		\begin{equation*}
			\textnormal{mld}(W; X, B) = \inf\{1 + a(E; X, B)\}
		\end{equation*}
	where $E$ runs through all divisors over $X$ which dominate $W$. We say $W$
	 is \textit{a log center} of $(X, B)$ if $\textnormal{mld}(X, B; W) < 1$. Note that $\textnormal{mld}(W; X, B) = 1 - \text{coeff}_W B$ if $W$ is an irreducible divisor on $X$.
	\end{definition}
	
	\begin{proposition}
		\label{prop: iso_spec}
		Let $S$ be a smooth variety over $k$. Let $f: (X, B)\to S$ and $f': (X', B')\to S$ be projective morphisms from lc pairs such that $B, B'$ are effective, and $K_{X} + B$ and $K_{X'} + B'$ are ample over $S$. Suppose there is an open dense subset $U\subseteq S$ together with a log isomorphism:
		\begin{equation*}
			\varphi_0: (X, B)|_{U} \to (X', B')|_{U}. 
		\end{equation*}
		Then $\varphi_0$ extends to a log isomorphism $\varphi: (X, B)\to (X', B')$ if every log center of $(X, B)$ (resp. $(X', B')$) is not included in $ X\smallsetminus f^{-1}(U)$ (resp. $X'\smallsetminus (f')^{-1}(U)$).
	\end{proposition}
	\begin{proof}
		Let $\nu: \Gamma\to X\times_S X'$ be a normalization of the closure for the graph $\varphi_0$ with projections $p: \Gamma\to X$ and $p': \Gamma\to X'$. Since the pairs $(X, B)$ and $(X', B')$ are lc with effective boundaries, we have $\text{mld}(W; X, B) < 1$ and $\text{mld}(W'; X', B')< 1$ for all $W\in \text{supp} \,B$ and $W'\in \text{supp}\, B'$. Recall that no log centers for $(X, B)$ and $(X', B')$ are included in $Z:= X\smallsetminus f^{-1}(U)$ and $Z':= X'\smallsetminus (f')^{-1}(U)$ respectively. Then we have $B_\Gamma:= (p^{-1})_* B = (p')^{-1}_*B'$, and write
		\begin{equation*}
			K_\Gamma + B_\Gamma = p^*(K_X + B) + E = (p')^*(K_{X'} + B') + E'
		\end{equation*}
		where $E$ is a $p$-exceptional effective $\R$-divisor such that $p_*(\text{supp} E)\subseteq Z$, and $E'$ is defined similarly. Hence,
		\begin{equation*}
			E - E' = (p')^*(K_{X'} + B') - p^*(K_X + B).
		\end{equation*}
		Therefore, $E-E'$ is $p$-nef and $-p_*(E- E')$ is effective. From the negativity lemma \cite[Theorem 11.60]{Kol23} it follows that $-(E-E')$ is effective. Similarly, $E'-E$ is effective, so $E = E'$. It implies that $p^*(K_X + B) = (p')^*(K_{X'} + B')$. We claim that the projections $p, p'$ from $\Gamma$ are finite. Indeed, if a curve $C\subseteq \Gamma$ is contracted by both $p, p'$ then it is contracted by the normalization morphism $\nu: \Gamma\to X\times_S X'$ due to the universal property of the fiber product. This is not possible because $\nu$ is finite. Hence, if there is a curve $C\subseteq \Gamma$ contracted by $p$, it can not be contracted by $p'$. Then we obtain a contradiction: $$0 = p^*(K_X + B)\cdot C = (p')^*(K_{X'} + B')\cdot C = (K_{X'} + B')\cdot (p')_* C > 0$$ because $K_{X'} + B'$ is ample over $S$. By the Zariski main theorem, the morphisms $p, p'$ are isomorphisms. Taking into account that $p^*(K_X + B) = (p')^*(K_{X'} + B')$, we conclude the proof.
	\end{proof}

	\section{Main results}
		\subsection{Smooth case}
		We refer the reader to \cite{CMP17} for the basics of Hodge theory and to \cite{Cat84} for an overview of various Torelli type problems. Let $S$ be a variety over $\Ko$. We write $S^\an$ (resp. $\Os^\an$) for the analytification of $S$ (resp. $\Os_S$).
		\begin{definition}
			\label{def: transcendental}
			Let $S$ be a regular variety. Let $\V = (\V_\Zi, F^\bullet\V_{\Os^\an})$ be a polarizable integral variation of pure Hodge structures ($\Zi$-VHS) on $S^\an$, and $F^n \V_{\Os^\an}$ be the deepest non-zero part of the Hodge filtration. Put $\V_\Q = \V_\Zi\otimes \Q$. By $(\V^\text{tr}_\Q, F^\bullet\V_{\Os^\an}^\text{tr})$ we denote the \textit{transcendental part of $\V$}, that is, a (unique) smallest polarizable rational Hodge substructure of $(\V_\Q, F^\bullet\V_{\Os^\an})$ for which $F^n \V_{\Os^\an}^\text{tr} = F^n \V_{\Os^\an}$. Note that $(\V^\text{tr}_\Q, F^\bullet\V_{\Os^\an}^\text{tr})$ is simple. We say $\V^\text{tr} = (\V^\text{tr}_\Zi, F^\bullet\V_{\Os^\an}^\text{tr})$ is the \textit{transcendental Hodge lattice for} $\V$, where $\V^\text{tr}_\Zi = \V^\text{tr}_\Q\cap \V_\Zi$.
		\end{definition}
		\begin{construction}
			\label{constr: period_maps_smooth}
			Let $f:\X\to S$ be a smooth projective morphism between regular varieties. Suppose $\X_s$ is $K$-torsion for all $s\in S$. We shall assume that there exists\footnote{We will see in the next subsection (\cref{lem: BB}) that such cover always exists \'etale locally.} an \'etale Galois cover $\pi: \X'\to \X$ such that $\Omega^d_{\X'_s} \approx \Os_{\X'_s}$ for all $s\in S$. Put $f' = f\circ \pi$.
			
			Let $s_0\in S$ be a closed point, and $d = \dim \X'/S$. Suppose $\A$ is an ample invertible sheaf on $\X'/S$. For every closed $s\in S$ there is a polarization $(\mathcal{Q}^\prim)_s$ on  $ (\V_\Zi^\prim)_s = H^d_\text{prim}(\X'_s, \mathbb{Q})\cap H^d(\X'_s, \mathbb{Z}) $ associated with the cup product on $H^d(\X'_s, \Zi)$ and the Lefschetz decomposition for $H^d(\X'_s, \Q)$ relative to~$\A_s$. By \cite[Corollary 2.3.5]{CMP17},  $(\mathcal{Q}^\prim)_s$ can be extended to all of $(\V_\Zi^\full)_s = H^d(\X'_s, \Q)\cap H^d(\X'_s, \Zi)$. By $(\Qp^\full)_s$ we denote the extended polarization, and $(\Qp^\tr)_s$ its restriction to $(\V^\text{tr}_\Zi)_s$.
			
			In the standard way \cite{Gri70}, we define $\V^\full =  (\V_\Zi^\full, F^\bullet \V_{\Os^\an}^\full, \Qp^\full)$ the polarized $\Zi$-VHS weight $d$ associated with $\V_{\Os^\an} = (R^df_*' \underline{\Ko})\otimes \mathcal{O}_{S^\an}$. We write $\Gamma^\full = \text{Aut}\left((\V_\Zi^\full)_{s_0}, \mathcal{Q}_{s_0}^\full\right)$ for the arithmetic monodromy group that acts naturally on the period domain $\D^\full$. By $\Phi^\full: S^\an\to \Gamma^\full\backslash\D^\full$ we denote the corresponding period map, which is a holomorphic map between analytic spaces \cite[Lemma-Definition 4.6.3]{CMP17}. The period maps $\Phi^\prim: S^\an \to \Gamma^\prim \backslash \D^\prim$ and $\Phi^\text{tr}: S^\an \to \Gamma^\tr\backslash \D^\tr$ are defined similarly.
		\end{construction} 
			  The following theorem asserts that the variations of Hodge structures  $\V^\tr\hookrightarrow \V^\prim \hookrightarrow \V^\full$  capture isomorphism types of fibers in a given family of K-torsion varieties. Furthermore,  the theorem provides a sufficient condition when two fibers of a given family are isomorphic.
		\begin{theorem}
			\label{th: main_smooth}
			Keep the notation and assumptions of \cref{constr: period_maps_smooth}. Let $P^\tr$ (resp. $P^\prim, P^\full$) be a connected component of $(\Phi^\tr)^{-1}(\star)$ (resp. $(\Phi^\prim)^{-1}(\star), (\Phi^\full)^{-1}(\star)$) for some point $\star$ in the corresponding period domain. Then $\X_s\approx \X_t$ for all (closed) points $s,t$ in $P^\tr$ (resp. in $P^\prim, P^\full$).
		\end{theorem}
		\begin{proof}
			Fix $P\in \{P^\tr, P^\prim, P^\full\}$. By \cite[Theorem 1.1]{BBT23}, $P$ is an algebraic closed subset of $S^\an$. Let $P_1, P_2 \subseteq S^\an$ be intersecting irreducible components. If $\X_{s_1}\approx \X_{t_1}$ for all $s_1,t_1\in P_1$ and $\X_{s_2}\approx \X_{t_2}$ for all $s_2, t_2\in P_2$ then $\X_s\approx \X_{r} \approx \X_{t}$ for all $s\in P_1, r\in P_1\cap P_2, t\in P_2$. Hence we shall assume that $P$ is irreducible. Let $\nu: \widetilde{P}\to P$ be (a normalization composed with) a resolution of singularities, and $\widetilde{\X} = \X\times_P \widetilde{P}$. Then $\widetilde{\X}_s \cong \X_{\nu(s)}$ for all (closed) points $s\in \widetilde{P}$ as the base field $\Ko$ is algebraically closed. It is clear that the family $\widetilde{\X}/\widetilde{P}$ satisfies the assumptions of \cref{constr: period_maps_smooth}. Hence we can assume that $P$ is a regular irreducible algebraic closed subset of $S$. To ease the notation, we denote the new families $\X_P/P, \X'_P/P$ by $\X/S, \X'/S$, i.e $P = S$.
			
			Now, we prove that $\dim H^1(\X'_s, T_{\X'_s})$ is independent of $s\in S$, and $\Ker(\kappa'_s) = T_{S, s}$ for all $s\in S$, where $\kappa'_s: T_{S, s}\to H^1(\X'_s, T_{\X'_s})$ is the Kodaira--Spencer map for $\X'/S$. Let $\lambda^\full_s: H^1(\X'_s, T_{\X'_s})\to \text{Hom}\left(H^{d,0}(\X'_s), H^{d-1, 1}(\X'_s)\right)$ be the natural map given by the composition of the cup product $H^1(\X'_s, T_{\X'_s})\otimes H^0(\X'_s, \Omega^d_{\X'_s})\to H^1(\X'_s, T_{\X_s}\otimes \Omega^d_{\X'_s})$ with the contraction $H^1(\X'_s, T_{\X_s}\otimes \Omega^d_{\X'_s})\to H^1(\X'_s, \Omega^{d-1}_{\X'_s})$. Let $$d_s\Phi^\full: T_{S, s}\to  \bigoplus_{p+q = d}\text{Hom}\left(H^{p,q}(\X'_s), H^{p-1, q+1}(\X'_s)\right)$$ be the differential of the period map $\Phi^\full$ at a closed point $s\in S$. Set $$\sigma^\full_s\colon T_{S, s}\to \mathrm{Hom}\left(H^{d,0}(\X'_s), H^{d-1, 1}(\X'_s)\right)$$ to be its bottom graded piece. Clearly, $\Ker (d_s\Phi^\full) \subseteq \Ker(\sigma_s^\full$). Then $\sigma_s^\full = \lambda_s^\full \circ \kappa'_s$ (cf. \cite[Theorem 5.3.4]{CMP17}). The maps $\lambda_s^\tr, \lambda_s^\prim$ and $\sigma_s^\tr, \sigma_s^\prim$ are defined similarly, and we have  $\sigma_s^\tr = \lambda_s^\tr \circ \kappa'_s$ and  $\sigma_s^\prim = \lambda_s^\prim \circ \kappa'_s$ for all $s\in S$. Since $S=P\in \{P^\tr, P^\prim, P^\full\}$, we have $T_{S,s}\in \{\Ker(\sigma_s^\tr), \text{Ker}(\sigma_s^\prim), \Ker(\sigma_s^\full)\}$. Note that $ (\V_\Ko^\tr)^{d,0}_s = H^{d,0}_\prim(\X'_s) = H^{d,0}(\X'_s)$ and $(\V_\Ko^\tr)^{d-1,1}_s\subseteq H^{d-1, 1}_\prim(\X'_s)\subseteq H^{d-1, 1}(\X'_s)$ by \cref{constr: period_maps_smooth}. Then
			we obtain the following commutative diagram
			$$
			\xymatrix{
				& &\text{Hom}\left(H^{d,0}(\X'_s), H^{d-1, 1}(\X'_s)\right)\\
				T_{S,s} \ar@{->}[r]^(0.35){\kappa'_s} \ar@/^2.5pc/[urr]^(0.3){\sigma^\full_s} \ar@/_2.5pc/[drr]_(0.3){\sigma^\tr_s} & H^1(\X'_s, T_{\X_s'}) \ar@{->}[r]^(0.4){\lambda^\prim_s} \ar@{->}[ur]^(0.45){\lambda^\full_s} \ar@{->}[dr]_(0.45){\lambda^\tr_s}& \text{Hom}\left(H^{d,0}(\X'_s), H^{d-1, 1}_\prim(\X'_s)\right) \ar@{}[u]|-*[@]{\subseteq}\\
				&&\text{Hom}\left(H^{d,0}(\X'_s), (\V_\Ko^\tr)^{d-1, 1}(\X'_s)\right). \ar@{}[u]|-*[@]{\subseteq}
			}
			$$
			From $\Omega_{\X'_s}^d\approx \Os_{\X'_s}$ it follows that $\lambda_s^\full$ is an isomorphism (infinitesimal Torelli theorem). Therefore, $\Ker(\kappa'_s) = \Ker(\sigma_s^\full)$. Moreover, the commutativity of  above diagram implies that $\Ker(\sigma^\tr_s) = \Ker(\sigma^\prim_s) = \Ker(\sigma^\full_s)$. Hence, $\Ker(\kappa'_s) = T_{S, s}$ for all closed points $s\in S$. Moreover, the dimension $\dim H^1(\X'_s, T_{\X'_s}) = h^{d-1, 1}(\X'_s)$ is independent of a point $s\in S$.
			
			Finally, we prove that $\dim H^1(\X_s, T_{\X_s})$ is independent of $s\in S$, and $\Ker(\kappa_s) = T_{S, s}$ for all $s\in S$, where $\kappa_s: T_{S, s} \to H^1(\X_s, T_{\X_s})$ is the restriction of the Kodaira--Spencer map for $\X/S$. Recall that $\pi: \X'\to \X$ is an \'etale Galois cover, and $\pi_s: \X_s'\to \X_s$ its induced cover. Let $G$ denote the deck group for $\X'/\X$ over $S$. Then $T_{\X_s}\cong \left((\pi_s)_* T_{\X'_s}\right)^{G}$ for all points $s\in S$. This makes commutative the following diagram
			$$
				\xymatrix{
					T_{S, s} \ar@{->}[r]^(0.35){\kappa'_s}  \ar@{->}[dr]_(0.5){\kappa_s} & H^1(\X'_s, T_{\X'_s})\\
					& H^1(\X_s, T_{\X_s}) \ar@{^{(}->}[u]^(0.35){}.
				}
			$$
			
			Therefore, $\Ker(\kappa_s) = \Ker(\kappa'_s) = T_{S, s}$ for all closed points $s\in S$. Consider the natural surjective map $\pi_*T_{\X'/S}\to T_{\X/S}$ locally given by $v\mapsto\frac{1}{|G|}\sum_{g\in G}gv$, and denote its kernel by $\mathcal{K}$. Then $\pi_*T_{\X'/S} \approx \mathcal{K}\oplus T_{\X/S}$. Hence the semicontinuity theorem \cite[Theorem 12.8]{Har77} together with constancy of $\dim H^1(\X_s', T_{\X'_s})$ implies that $\dim H^1(\X_s, T_{\X_s})$ is independent of $s\in S$. Hence, $R^1f_*T_{\X/S}$ is locally free, and the natural map $R^1f_*T_{\X/S}\otimes \kappa(s)\to H^1(\X_s, T_{\X_s})$ is an isomorphism for all $s\in S$ by \cite[Corollary 12.9]{Har77}, where $\kappa(s)$ is the residue field at $s\in S$. Thus, the (global) Kodaira--Spencer map $T_S\to R^1f_*T_{\X/S}$ is the zero map. From the deformation theory it follows that $\X/S$ is isotrivial, as required.
		\end{proof}

		\begin{remark}
			One can consider a variation $\V$ of integral polarized Hodge structures coming directly from $\X/S$. Then the analogue of \cref{th: main_smooth} does not hold in general (see \cref{ex: Enriques}) for the period map associated with $\V$ unless all fibers $\X_s$ are K-trivial, i.e $\Omega_{\X_s}^d\approx \Os_{\X_s}$.
		\end{remark}
		\begin{example}\label{ex: Enriques}
			Let $X$ be a smooth projective Enriques surface, that is, $\omega_X^{\otimes 2} \approx \Os_X$ and $H^1(X, \Os_X) = H^2(X, \Os_X) = 0$. Then $H^2(X, \Ko) = H^{1,1}(X)$. This implies that the Kuranishi family $\X/S$ for $X$ has constant period maps. To construct a similar example in the projective category, one can use the fact that all projective Enriques surfaces admit a bounded polarization \cite[Theorem 6.3]{FHS25}.
		\end{example}
		
		\begin{remark}
			Let us outline a different approach for \cref{th: main_smooth} via the specialization argument. For simplicity we assume that $P^\full = S$. Then there is a K-torsion variety $X$, and an open dense subset $U\subseteq S$ such that $\X_s\approx X$ for all closed points $s\in U$ by \cite[Theorem 2.2]{Amb05}. Since the Kodaira dimension $\kappa(\X_s)$ is zero for all points $s\in S$, we have $\X_s$ is not uniruled. In particular, $\X_s$ is not ruled. Now fix a closed point $o\in S$, a general  regular curve $C\subseteq S$ passing through $o$, and consider a divisorial valuation ring $R = \Os_{C, o}$ with its generic point $\eta$. After an \'etale base change, we can assume that $\X_\eta \approx X\times_\Ko \kappa(\eta)$. By \cite[Theorem 2]{MM64} the family $\X_\eta$ has at most one smooth extension over $\text{Spec}\, R$ where the central fiber $\X_o$ is not ruled. This implies that $\X_o \approx X$.
		\end{remark}

		\subsection{General case}\label{sec: general_case}
		Now, we generalize ideas provided in the previous section. Since projective families of varieties with singularities  \textit{apriori} carry only natural mixed Hodge structures, we introduce several technical but useful constructions to obtain a variation of pure Hodge structures.
		
		\begin{definition}
			Suppose $(X, B)$ is a pair. Then the $\R$-divisor $B$ can be uniquely written as a difference of two effective $\R$-divisors $B_{\ge 0}$ and $B_{\le 0}$ without common components. We say that $N = \lceil B_{\le 0} \rceil$ is the \textit{negative part} of $B$, and $F = B - \lfloor B \rfloor$ is the \textit{fractional part} of $B$. In general, $F$ and $N$ may have common components. Then we can write
			$
			B = \lfloor B_{\ge 0} \rfloor + F - N.
			$
		\end{definition}
		
		\begin{construction}
			\label{const: covering_trick}
			Let $(X, B_X)$ be a klt log pair such that $B_X$ is a $\Q$-divisor and $K_X + B_X \sim_\Q 0$. Consider a log resolution $g: (Y, B_Y) \to (X ,B_X)$, where $K_Y + B_Y = g^*(K_X + B_X)$. Write $B_Y = F_Y-N_Y$, where $N_Y$ (resp. $F_Y$) is the negative (resp. fractional) part of~$B_Y$. Suppose $m\in \Zi_{>0}$ is the minimal integer such that $mF$ is integral, and $m(K_Y + B_Y)\sim~0$. Set $L = \mathcal{O}_Y(N_Y-K_Y)$. A choice of the section $s\in H^0(Y, L^{\otimes m})$ such that $mF_Y = (s = 0)$ determines a normalized cyclic cover $\pi^1:~(Z, B_Z) \to (Y, B_Y)$ of degree $m$ ramified over $F$. As always, $K_Z + B_Z = (\pi^1)^*(K_Y + B_Y)$. We shall say that $\pi: (Z, B_Z) \to (Y, B_Y)$ is an \textit{index-1 cover}. Moreover, $(\pi^1)_* \omega_Z \cong \bigoplus_{i=0}^{m-1}\omega_Y\otimes L^{\otimes i}(-\lfloor iF_Y \rfloor)$ (cf. \cite[Section 8.10]{Kol07}). In particular, $H^0(Y, \mathcal{O}_Y(N_Y))$ is naturally a direct summand of $H^0(Z, \omega_Z)$. The latter is the deepest non-zero piece of the Hodge filtration on $H^{\dim Z}(Z, \Ko)$.
		\end{construction}
		\begin{remark}
			\label{remark: mixed_are_pure}
			Since $(Y, B_Y)$ is log smooth, $Z$ has (at worst) quotient singularities (see \cite[Section 8.10.4]{Kol07}). Then the canonical mixed Hodge structure on $H^n(Z, \Q)$ is pure by \cite{Ste76}. 
		\end{remark}
		
		\begin{definition}
			\label{def: CY-VHS}
			Let $S$ be a regular variety. Let $\V = (\V_\Zi, F^\bullet\V_{\Os^\an})$ be a polarizable $\Zi$-VHS on $S^\an$. We say $\V$ is a \textit{Calabi--Yau variation} if the deepest non-zero part $F^n\V_{\Os^\an}$ of the Hodge filtration has rank one. In this case, we call $\V$ a polarizable CY $\Zi$-VHS, and $\Ho_S = F^n \V_{\Os^\an}\cong \gr^n_F\V_{\Os^\an}$ the Hodge line bundle.
		\end{definition}
		\begin{definition}
			\label{def: lc-trivial}
			Let $(\X, \B)$ be a log pair, $\B$ is a $\Q$-divisor, and $S$ a normal variety. A proper contraction $f: \X \to S$ is called \textit{lc-trivial} if
			\begin{enumerate}
				\item 
				$(\X_\eta, \B_\eta)$ is a lc log pair, where $\eta$ is the generic point of $S$.
				\item 
				$K_\X + \B \sim_{\Q, S} 0$, i.e. $K_\X + \B \sim_\Q f^*L$ for some $\Q$-Cartier $\Q$-divisor $L$ on $S$.
				\item 
				$\textnormal{rk}\, f_*\Os\left(\lceil \mathbf{A}^*(\X, \B)\rceil\right) = 1$,
			\end{enumerate}
			where $\mathbf{A}^*(\X, \B)$ is the b-divisor defined by taking its trace on any birational model $g: \Y\to \X$, that is, $\mathbf{A}^*(\X, \B)_{\Y} := K_{\Y} - g^*(K_\X + \B) + \sum_{a(E; \X, \B) = 1} E.$
		\end{definition}

		\begin{construction}
			\label{const: main}
			Let $f_\X: (\X, \B_\X) \to S$ be a projective lc-trivial fibration of relative dimension $d\in \Zi_{>0}$. We shall assume that $(\X/S, \B)$ is a projective elementary klt family. Consider a log resolution $g: (\Y, \B_\Y) \to (\X, \B_\X)$, and $E_\Y$ be the reduced exceptional divisor. We assume that $(\Y/S, \B_\Y)$ (resp. $(\Y/S, E_\Y)$) is a projective elementary klt families (resp. an elementary family). Similarly, we assume both $(\Y, \B_\Y), (\Y, E_\Y)$ are log smooth over $S$. As in \cref{const: covering_trick}, we write $\B_\Y = F_\Y - N_\Y$, and consider $\pi^1: (\Z, \B_\Z)\to (\Y, \B_\Y)$ an index-1 cover of degree $m\in \Zi_{>0}$. Choose a $\mu_m$-equivariant resolution of singularities $\nu: \widetilde{\Z} \to \Z$:
			$$
			\xymatrix{
				(\Y, \B_\Y)\ar@{->}[d]_{g} \ar@/_2.5pc/[dd]_{f_\Y} &(\Z, \B_\Z)\ar@{->}[l]_{\pi^1} \ar@{->}[ddl]^(0.4){f_\Z}&(\widetilde{\Z}, \B_{\widetilde{\Z}})\ar@{->}[l]_(0.5){\nu}  \ar@{->}[ddll]^(0.4){f_{\widetilde{\Z}}}\\
				(\X, \B_\X) \ar@{->}[d]_{f_\X} & &\\
				S & &.
			}
			$$
			
			We shall assume that the induced morphism $f_{\widetilde{\Z}}$ is smooth. The sheaf $(\pi\circ\nu)_*\omega_{\widetilde{\Z}}$ is $\mu_m$-equivariant, and it admits a decomposition into $\mu_m$-isotypic components:
			\begin{equation*}
				(\pi\circ\nu)_*\omega_{\widetilde{\Z}} = \pi_* \omega_{\Z} \cong \bigoplus_{i=0}^{m-1} \pi_*(\omega_\Z)_{\chi^i} \cong \bigoplus_{i=0}^{m-1}\omega_{\Y}\otimes \Li^{\otimes i}\left(-\lfloor i F_\Y\rfloor\right),
			\end{equation*}
			where $\chi$ is a generator of the character group $\hat{\mu}_m = \text{Hom}(\mu_m, \Ko^\times)$. In particular, we get the direct summand for $i=1$:
			\begin{equation*}
				\Os_{\Y}(N_\Y) \cong (\pi\circ \nu)_*\left(\omega_{\widetilde{\Z}}\right)_\chi \hookrightarrow (\pi\circ \nu)_* \omega_{\widetilde{\Z}}.
			\end{equation*}
			Let $\V_\Zi = \left(R^d(f_{\widetilde{\Z}})_*\underline{\Zi}\right)_\textnormal{tf}$, and $\V_\Ko = \V_\Zi\otimes \Ko$. By $\V_{\chi^i}$ we denote the $\chi^i$-isotypic component of $(\V_\Ko, F^\bullet \V_{\Os^\an})$, where $i = 1, 2, \dots, m$. Then after shifting the Hodge filtration $$\V_\chi(-2, 2)\oplus \left(\bigoplus_{i\neq \pm 1 \text{ mod } m} \V_{\chi^i}\right)\oplus \V_{\chi^{-1}}(2, -2)$$ we a obtain polarizable CY $\Zi$-VHS $\V$ of Hodge-level $d+4$ whose underlying local system is $\V_\Zi$.
			By construction, the Hodge line bundle for $\V$ is $\Ho = (f_\Y)_*\Os_\Y(N_\Y)$.
		\end{construction}

		\begin{definition}
			\label{def: CY_period_map}
			Let $S$ be a regular variety, and $s_0\in S$ be a closed point. Let $\V = (\V_\Zi, F^\bullet\V_{\Os^\an})$ be a polarizable CY $\Zi$-VHS on $S^\an$. Choose a polarization $\Qp$ on $\V^\text{tr}$. Put $V_\Zi = (\V_\Zi^\text{tr})_{s_0}$, $Q = \Qp_{s_0}$. Let $\Gamma^\CY = \text{Aut}(V_\Zi, Q)$ be the \textit{arithmetic monodromy group}, that is, a subgroup of $\text{GL}(V_\Zi)$ preserving the polarization $Q$. By $\Phi^\text{CY}: S^\an\to \Gamma^\CY\backslash \D^\text{CY}$ we denote the period map associated with $\V^\text{tr}, \Qp$, and $s_0$. By \cite[Lemma-Definition 4.6.3]{CMP17}, $\Phi^\text{CY}$ is a holomorphic map between analytic spaces. We shall say that $\Phi^\text{CY}$ is a \textit{CY period map} associated with the polarizable CY $\Zi$-VHS $\V$.
		\end{definition}
		\begin{remark}
			It is readily seen that CY period maps are independent of the choice $s_0\in S$ up to isomorphism. However, CY period maps depend on the choice of a polarization $\Qp$ on $\V^\text{tr}$ apriori.
		\end{remark}

		Recall that for every projective klt log pair $(X, B)$, the underlying variety $X$ always admits \cite[Section 2.2]{MW25} an \textit{MRC fibration}, that is, a dominant rational map $\psi: X\dashrightarrow Y$ to a normal variety $Y$ such that a very general rational curve in $X$ is included in the fiber $\psi^{-1}(y)$ for some $y\in Y$, a general fiber is rationally connected, and the indeterminacy locus is not dominant over $Y$.
		
		The following theorem is a singular version of the Beauville--Bogomolov decomposition \cite{Bog74, Bea83}. We refer an interested reader to \cite{BGL22, MW25} and references therein for history of the question. Recall that a \text{cover} is a finite surjective morphism. A cover $\pi: X'\to X$ is \textit{quasi-\'etale} if there is an open dense subset $U\subseteq X$ such that $\pi$ is \'etale over $U$ and $\text{codim } X\smallsetminus U\ge 2$. 
		\begin{theorem}
			\label{th: BB-klt}
			Let $(X, B)$ be a projective klt $0$-pair with $\Q$ coefficients. Then there exists a quasi-\'etale Galois cover $\pi: X'\to X$ such that $$(X', B') = (R, B_{R})\times_\Ko Ab\times_\Ko \prod_i CY_i \times_\Ko \prod_j S_j$$
			where $B' = \pi^* B$, and
			\begin{itemize}
				\item 
					$Ab$ is an abelian variety.
				\item 
					\begin{itemize}
						\item 
							 $R$ is a fiber of the MRC fibration for $X'$,
						\item 
							$(R, B_R)$ is a projective klt $0$-pair where $B_R = B'|_R$.
					\end{itemize}
				\item 
					 $CY_i$ is an irreducible Calabi-Yau variety:
					 \begin{itemize}
					 	\item 
					 		$CY_i$ has (at worst) canonical singularities,
					 	\item 
					 		$K_{CY_i}\sim 0$,
					 	\item 
					 		for all quasi-\'etale covers $CY_i'\to CY_i$ we have $$H^0(CY_i', \Omega^{[n]}_{CY'_i}) = 0$$ for all $0<n<\dim CY_i'$.
					 \end{itemize}
				\item 
					$S_i$ is a strict irreducible symplectic variety:
					\begin{itemize}
						\item 
							$S_i$ has (at worst) canonical singularities,
						\item 
							$K_{S_i}\sim 0$,
						\item 
							for all quasi-\'etale covers $S'_i\to S_i$ we have $$H^0(X, \Omega_{S'_i}^{[2]}) = \Ko \sigma \text{ and } H^0(X, \Omega_{S'_i}^{[1]}) = 0$$ for some $2$-form $\sigma$ non-degenerate on $(S_i')^{\textnormal{reg}}$,
						\item 
							the \'etale fundamental group ${\pi}_1^\text{\'et}(S_i^\textnormal{reg})$ of its regular locus is trivial.
					\end{itemize}
			\end{itemize}
		\end{theorem}
		\begin{lemma}
			\label{lem: BB}
			Keep the notation and assumptions of \cref{const: main}. Then for every closed point $o\in S$ there is an \'etale neighborhood $U$ of $o$, and a quasi-\'etale Galois cover $\pi: \X'\to \X_U$ such that
			$$(\X', \B_{\X'}) = (\Ra, \B_\Ra)\times_\Ko \A b\times_\Ko \prod_i \Cy_i\times_\Ko \prod_j \Sy_j$$
			where for every point $s\in U$ the decomposition $(\X'_s, \B_{\X'_s}) = (\Ra_s, \B_{\Ra_s})\times_\Ko \A b_s\times_\Ko \prod_i (\Cy_i)_s\times_\Ko \prod_j (\Sy_j)_s$ is as in \cref{th: BB-klt}.
		\end{lemma}
		\begin{proof}
			Let $X = \X_o$, and consider a quasi-\'etale Galois cover $\pi_0: X'\to X$ as in \cref{th: BB-klt}. From our assumptions on the log resolution $g: (\Y, \B_\Y) \to (\X, \B_\X)$ it follows $\Y$ and $\Y\smallsetminus E_\Y$ are topologically locally trivial families over $S$. Then $\X^\text{reg}/S$ is so. This implies that for any sufficiently small contractible analytic neighborhood $O\subseteq S^\an$ of $o\in S$ the natural inclusion $(\X^\text{reg})_o^\an \hookrightarrow \X_O^\an$ induces an isomorphism $\pi_1\left((\X^\text{reg}\right)^\an_o) \to \pi_1(\X_O^\an)$. Again from our assumptions on $\Y/S, E_\Y/S$ it follows that $(\X_s)^\text{reg} = (\X^\text{reg})_s$ for all points $s\in S$. Hence $\pi_0: X'\to X$ corresponds to a quasi-\'etale cover $\pi^\an: (\X^\an_O)'\to \X_O^\an$ such that $\pi_o^\an =\pi_0^\an$ (cf.  \cite[Section 3.6]{BGL22}).
			
			Now, we show that the cover $\pi^\an: (\X^\an_O)'\to \X_O^\an$ induces a cover over $\mathfrak{X}: = \X\times_S \text{Spec}\, \widehat{R}$ where $R = \Os_{S, o}$ (cf. \cite[Proposition 7.4]{DG18}). Let $\mathfrak{m}\subset R$ be the maximal ideal, and set $R_n = R/\mathfrak{m}^n, \ \X_n = \X\times_S \Spec R_n$. Since the analytification commutes with fiber products and $(\Spec R_n)^\an\hookrightarrow S^\an$ factors through $O$, we have $\X_n^\an \cong \X^\an \times_O (\Spec R_n)^\an$. Define analytic sheaves of $\Os_{\X^\an_O}$-algebras $\mathscr{A}^\an := \pi^\an_*\Os_{(\X^\an_O)'}$ and $\mathscr{A}_n^\an := \mathscr{A}\otimes_{\Os_{\X^\an_O}} \Os_{\X_n^\an}$. Then for all $n\ge 1$ and $(\X_n^\an)' := (\X^\an_O)'\times_{\X^\an_O}\X^\an_n$ we obtain a Cartesian square:
			
			$$
				\xymatrix{
					(\X_n^\an)' \ar@{->}[r] \ar@{->}[d] & (\X_{n+1}^\an)' \ar@{->}[d] \\
					\X_n^\an \ar@{->}[r] & \X_{n+1}^\an.
				}
			$$
			
			Since for every $n\ge 1$ the scheme $\X_n$ is proper over $\Ko$, the coherent $\Os_{\X_n^\an}$-algebra $\mathscr{A}^\an_n$ is the analytification of a coherent $\Os_{\X_n}$-module $\mathscr{A}_n$ by \cite[Expos\'e XII, Th\'eor\`eme 4.4]{SGA1}, i.e. the GAGA principle. In fact, $\mathscr{A}_n$ is an $\Os_{\X_n}$-algebra. Then we can define $\X_n' := \Spec_{\X_n} \mathscr{A}_n$ which comes together with the cover $\pi_n: \X_n'\to \X_n$ such that $(\pi_n)_*\Os_{\X_n'} = \mathscr{A}_n$. By construction, we have a compatible system $(\mathscr{A}_n, \theta_n)$ of $\Os_\mathfrak{X}$-algebras where $\theta_n: \mathscr{A}_n\otimes_{\Os_{\X_{n+1}}}\Os_{\X_n}\to \mathscr{A}_{n+1}$ is an isomorphism of $\Os_{\X_{n+1}}$-algebras. According to the Grothendieck Existence Theorem \cite[Th\'eor\`eme 5.1.4]{EGAIII1}, there exists a coherent $\Os_\mathfrak{X}$-algebra $\mathscr{A}$ on $\mathfrak{X}$ such that $\mathscr{A}\otimes_{\Os_\mathfrak{X}} \Os_{\X_n} = \mathscr{A}_n$ for all $n\ge 1$. Then we set $\widehat{\pi}: \mathfrak{X}'\to \mathfrak{X}$ for $\mathfrak{X}': = \Spec_\mathfrak{X} \mathscr{A}$. Then by Artin's algebraic approximation theorem \cite[Corollary 2.2]{Art69}, there exists an \'etale neighborhood $U$ of $o\in S$ and a finite cover $\pi: \X'\to \X_U$ such that $\pi_o = \widehat{\pi}_o$ but $\widehat{\pi}_o = \pi_0$. In addition, we can assume that $\X'\to \X_U$ is quasi-\'etale. To conclude the proof, one needs to show that the decomposition for $\X_o' = X'$ spreads out over $U$ possibly after shrinking $U$ around~$o$. This is done in the proof for \cite[Proposition 2.28]{FX26}.
		\end{proof}
		
		\begin{theorem}
			\label{th: main}
			Keep the notation and assumptions of \cref{const: main}. Let $P$ be a connected component of $(\Phi^\CY)^{-1}(\star)$ for some point $\star\in \Gamma^\CY\backslash\D^\CY$. Then $(\X_s, \B_{\X_s}) \approx (\X_t, \B_{\X_t})$ for all (closed) points $s,t\in P$.
		\end{theorem}
		\begin{proof}
			Let us remember that $(\Y, \B_\Y)$ is a projective elementary klt family, and $(\Y, \B_\Y)$ is log smooth over $S$ by \cref{const: main}. By \cite[Theorem 1.1]{BBT23} , $P$ is a closed algebraic subset of $S$. As in the beginning of proof for \cref{th: main_smooth}, we can assume that $P$ is a irreducible. Let $\nu: {P^\nu}\to P$ be a normalization, and $({\Y^\nu}, \B_{{\Y^\nu}}) = (\Y, \B_\Y)\times_S {P^\nu}, ({\X^\nu}, \B_{{\X^\nu}}) = (\X, \B_\X)\times_S {P^\nu}$. Since $\Phi^\CY(P) = \{\star\}$, we conclude that  the moduli part $M_{P^\nu}$ of the lc-trivial fibration $(\X^\nu, \B_{\X^\nu}) \to P^\nu$ is trivial, i.e. $M_{P^\nu}\sim _\Q 0$. Note that the boundary divisor $B_{P^\nu}$ is zero. Then by \cite[Theorem 4.7]{Amb05} there is a quasi-\'etale Galois cover $\tau: P'\to P^\nu$, an open dense subset $U\subseteq P'$, a log pair $(X, B_X)$, and an isomorphism $$(\X^\nu, \B_{\X^\nu})\times_{P^\nu} U \approx (X, B_X)\times_\Ko U \text{ over } U.$$
			Let $\mu: \widetilde{P}\to P'$ be a resolution of singularities, and $(\widetilde{\X}, \B_{\widetilde{\X}}), (\widetilde{\Y}, \B_{\widetilde{\Y}})$ be the corresponding new families over $\widetilde{P}$. Note that $(\widetilde{\Y}, \B_{\widetilde{\Y}})$ is a projective elementary klt family, and $(\widetilde{\Y}, \B_{\widetilde{\Y}})$ is log smooth over $\widetilde{P}$. By construction, the map $\alpha = \nu\circ\tau\circ\mu\colon\widetilde{P} \to P$ is surjective, and $(\widetilde{\X}_s, \B_{\widetilde{\X}_s}) \approx ({\X}_{\alpha(s)}, \B_{{\X}_{\alpha(s)}})$ because the base field $\Ko$ is algebraically closed. Consequently, we can denote the new families by $(\X/S, \B_{\X}), (\Y/S, \B_{\Y})$ to simplify the notation. It follows that the generic fibers of $({\Y}/S, \B_{{\Y}})$ and $({Y}, B_{{Y}})\times_\Ko S$ are 0-equivalent where $({Y}, B_{{Y}})\to (X, B)$ is any log resolution. Therefore, $({\Y}, \B_{{\Y}_s})\sim_\text{bir} ({Y}, B_{{Y}})$ for all closed points $s\in S$ according to \cref{th: bir_spec}. Hence, $({\X}_s, \B_{{\X}_s})\sim_\text{bir} (X, B)$ by \cref{def: 0-class}.
			
			However, we claimed that even stronger statement holds: $(\X_s, \B_{\X_s}) \approx (\X_t, \B_{\X_t})$ for all (closed) points $s,t\in S$. To prove it, we fix a closed point $o\in S$, and consider an \'etale neighborhood $U$ for $o$ as in \cref{lem: BB} for which there is a quasi-\'etale Galois cover $\X'\to {\X}_{U}$ inducing a singular Beauville-Bogomolov decomposition fiber-wise.  In other words, $$(\X', \B_{\X'}) = (\Ra, \B_\Ra)\times_\Ko \A b\times_\Ko \prod_i \Cy_i\times \prod_j \Sy_j$$
			where for every point $s\in U$ the decomposition $(\X'_s, \B_{\X'_s}) = (\Ra_s, \B_{\Ra_s})\times_\Ko \A b_s\times_\Ko \prod_i (\Cy_i)_s\times \prod_j (\Sy_j)_s$ is as in \cref{th: BB-klt}. Before we proceed further, we assume that the family $(\X'/U, \B_{\X'})$ is trivial over some open dense subset again by \cite[Theorem 4.7]{Amb05}, and let $$(X', B_{X'}) = (R, B_R)\times_\Ko Ab\times_\Ko \prod_i CY_i \times_\Ko S_j$$ be the induced decomposition for a general fiber of $(\X'/U, \B_{\X'})$. We claim that the families $(\Ra/U, \B_\Ra), \A b/U, \Cy_i/U, \Sy_j/U$ are isotrivial for all $i, j$. First, the isotriviality for $\A b/U$ follows from \cref{th: main_smooth}. Second, for every $s\in U$ the variety $\Ra_s$ is rationally connected with (at worst) klt singularities. This implies that $h^1(\Ra_s, \Os_{\Ra_s}) = 0$ by \cite{KM97}, and $\text{Pic}^0(\Ra_s) = 0$. Then $\text{Pic}(R\times_\Ko U) \cong \text{Pic}(R)\times \text{Pic}(U)$ by \cite[Chapter III, Exercise 12.6]{Har77}. Therefore, any invertible sheaf from the generic fiber $\Ra_\eta$ of $\Ra/U$ could be uniquely extended to an invertible sheaf on $R\times_\Ko U$ up to isomorphism. Choose an ample invertible sheaf  $\A$ on $\Ra/U$, and let $A\times_\Ko U$ be an extension of $\A_\eta$ to $R\times_\Ko U$. Similarly, we choose a general global section $\mathcal{H}$ of $\A^{\otimes m}$ for $m\gg 1$ which is regular and irreducible by Bertini's theorem, and consider the extension $H\times_\Ko U$ of $H_\eta$. Then by \cref{prop: iso_spec} for any general regular curve passing through $o\in S$ there is an isomorphism $(\Ra_C, \B_{\Ra_C} + \mathcal{H})\approx (R, B_R + {H})\times_\Ko C$. This implies that $(\Ra_o, \B_{\Ra_o})\approx (R, B_R)$. Third, for each point $s\in S$ and for all varieties $V\in\{(\Cy_i)_s, (\Sy_j)_s\}_{i,j}$ the vanishing of $h^1(\Os_V)$ follows from $h^0(\Omega^{[1]}_V) = 0$. Hence, \cref{prop: iso_spec} could be applied in the same way again. Therefore, after an \'etale base change we can assume that $(\X', \B_{\X'})\approx (X', B_{X'})\times_\Ko U$ over $U$, and $o\in U$.
			
			Finally, we use triviality of $(\X'/U, \B_{\X'})$ to deduce isotriviality for $(\X/U, \B_\X)$. Let $G$ be the deck group for $\X'/\X$, and $H = \textbf{Aut}(X')$ be group scheme such that $\textbf{Aut}(X')(\Ko) = \text{Aut}(X')$. Recall that every finite group in characteristic zero is reductive by Maschke's theorem. Then by \cite[Theorem 2]{Bri22} the functor $\mathbf{Hom}_\text{gr}(G, H)$ of morphisms of group schemes is represented by a smooth scheme $M$, and the morphism $H\times_\Ko M\to M\times_\Ko M$ given by $(h, \rho)\mapsto (h\rho h^{-1}, \rho)$ is smooth. By \cite[Remark 6.4]{Bri22}, the connected components of $M$ are exactly the $H^0$-orbits of $\Ko$-points, where $H^0:= \textbf{Aut}^0(X')$. By construction of $\X'/\X$, there is a morphism $U\to M$, i.e. a collection of maps $\rho_s: G\to \Aut(X')$. Its image is contained in a connected component of $M$. Hence, for every point $s\in U$ there is $h_s\in \Aut^0(X')$ such that $\rho_s = h_s\rho_o h_s^{-1}$. This implies that $\X_s\approx \X_o$ for all closed points $s\in U$, as required.  
		\end{proof}

		\subsection{Moduli space}\label{sec: moduli spaces}
		In this subsection we introduce a different approach for \cref{th: main} via the theory of moduli spaces leading to an independent results. We refer a reader to \cite{Kol23} that introduces all technical aspects of the KSBA theory. We follow the exposition in \cite[Section 7.3]{BFMT25} concerning compactifications of the moduli space for \textit{polarized klt 0-pairs}, that is, $(X, B; A)$ consists of a projective klt 0-pair $(X, B)$ with an ample invertible sheaf $A$ on $X$. Two polarized pairs $(X, B; A)$ and $(X', B'; A')$ are isomorphic if there is a log isomorphism $\varphi: (X, B) \to (X', B')$ such that $\varphi^* A' \approx A$ where the latter is an isomorphism of invertible sheaves. We shall assume that all boundaries $B$ are effective $\Q$-divisors.
		
		Fix $d, N\in \mathbb{Z}_{> 0}$, and $\mathbf{b}\in \Q^I$ for some $I\in \Zi_{\ge 0}$. Consider the $\mathbb{G}_m$-rigidified moduli problem for locally stable polarized\footnote{A polarization is the sheafification of a pre-polarization, i.e. an \'etale cover over the base together with relative ample invertible sheaves which agree on that cover \cite[Definition 8.40]{Kol23}.} families in the sense of Koll\'ar of polarized klt 0-pairs $(X, B; A)$ such that $\dim X = d$, $A^d = N$, and $B$ marked by $\mathbf{b}$. Then the corresponding functor $\mathcal{M}_{d, N, \mathbf{b}}$ is an algebraic stack over $\Ko$. When no ambiguity arises, we drop the index $\mathbf{b}$ from the notation. From \cite[Theorem 1.1]{Kol93} it follows that there is $n = n(d)\in \mathbb{Z}_{>0}$ such that $A^{\otimes n}$ is very ample. Hence, by the standard Hilbert--Chow scheme argument \cite[Proposition 2.15]{Ser26} all pairs $(X, B)$ can be embedded in some fixed projective space. Hence $\mathcal{M}$ is of finite type over $\Ko$. By \cite[Proposition 10.1]{PZ25}  the group $$\text{Aut}(X, B; A) := \{\varphi\in \text{Aut}(X): \varphi_*(B) = B, \ \varphi^*A \approx A\}$$ is finite. Then $\mathcal{M}_{d, N}$ is a separated Deligne--Mumford stack of finite type over $\Ko$ (cf. \cite[Theorem 8.1]{Kol23}). By \cite[Corollary 1.3]{KM97} the stack $\mathcal{M}_{d, N}$ admits a coarse moduli space $M_{d, N}$ which is a separated algebraic space of finite type. Next, there is a finite stratification of $\mathcal{M}_{d, N}^\text{red}$ (resp. $M_{d, N}$) by Zariski locally closed substacks $\mathcal{U}_i$ (resp. algebraic spaces $U_i$) analogues to families in \cref{const: main}. Hence, there are well-defined CY period maps $U_i\to \Gamma_i^{\CY}\backslash \mathbb{D}_i^\CY$. For any projective klt 0-pair $(X, B)$ we define the following sets:
		\begin{equation*}
			\begin{gathered}
				\Iso_{(X, B)}(M_{d, N}):= 
				\begin{Bmatrix}
					\ [(X', B'; A')]\in M_{d, N}(\Ko): & (X', B')\approx (X, B) \
				\end{Bmatrix},\\
				\textnormal{Wlc}_{(X, B)}(M_{d, N}) :=
				\begin{Bmatrix}
					\ [(X', B'; A')]\in M_{d, N}(\Ko): & (X', B')\sim_\text{bir} (X, B) \
				\end{Bmatrix}.
			\end{gathered}
		\end{equation*}
		\begin{theorem}
			\label{th: moduli_finite}
		Let $(X, B)$ be a projective klt 0-pair. Then the set $\textnormal{Wlc}_{(X, B)}(M_{d, N})$ is finite. In addition, every CY period map  $U_i\to \Gamma_i^{\CY}\backslash \mathbb{D}_i^\CY$ is quasi-finite.
		\end{theorem}
		\begin{proof}
			Recall that there is $n=n(d)\in \Zi_{>0}$ such that for every triple $[(X', B'; A')]$ the invertible sheaf $A'^{\otimes n}$ is very ample. By \cite[Theorem 5.5]{Ser26}, the set $\Cc$ of (log) isomorphism classes of all projective klt 0-pairs $(X', B')\sim_\text{bir} (X, B)$ admitting a very ample invertible sheaf of bounded degree is  finite. Let $(X_\al, B_\al)$ be pairs representing all classes in $\Cc$. Clearly, we have the following disjoint union of finitely many sets:
			\begin{equation*}
				\textnormal{Wlc}_{(X, B)}(M_{d, N}) = \bigsqcup_\al \Iso_{(X_\al, B_\al)}(M_{d, N}).
			\end{equation*}
			Next, by \cite[Theorem 5.6]{Ser26} for every $\al$ there finitely many ample invertible sheaves $A_{\al\beta}$ on $X_\al$ such that
			\begin{equation*}
				\Iso_{(X_\al, B_\al)}(M_{d, N})= \bigsqcup_\beta
				\begin{Bmatrix}
					\ [(X', B'; A')]\in M_{d, N}(\Ko): & \varphi: (X', B')\approx (X_\al, B_\al), \ \varphi_*A' \equiv A_{\al \beta} \
				\end{Bmatrix}.
			\end{equation*}
			Each set in the latter decomposition is finite by \cite[Section 7.3 (ii)]{BFMT25}, and the proposition follows. Let us reproduce the argument for completeness. To ease the notation, we omit the index $\al$. Consider its Albanese morphism $f: (X, B) \to \text{Alb}(X)$. Then there exists (by \cite[Theorem 4.8]{Amb05}) an  \'etale Galois cover $\tau: Ab\to \text{Alb}(X)$ such that $(X, B)\times_{\text{Alb}(X)} Ab \approx (F, B|_F)\times_\Ko Ab$ over $Ab$ where $(F, B|_F)$ is any (closed) fiber of $f$. It is clear that $Ab$ is an abelian variety. Let $G$ denote the deck group for $\tau$. Note that the automorphisms of $(F, B|_F)\times_\Ko Ab$ induced by translations of $Ab$ are $G$-equivariant, and they preserve the boundary. Hence, there is a morphism of groups $h: Ab\to \Aut(X, B)$. In fact, we have $h: Ab\to \Aut^0(X, B)$ because $Ab$ is connected. It is standard that the map $Ab\to \text{Pic}^0(X)$ given by $a\mapsto h(a)^*A\otimes A^{-1}$ is an isogeny for any fixed ample invertible sheaf $A$ on $X$. Hence, we have $[(X, B; A)] = [(X, B; A')] \in M_{d, N}$ if $A\sim_\text{alg} A'$. To conclude the proof it remains to use the fact \cite[Corollary 9.6.17]{FGA05} that the group $\text{Pic}^\text{tor}(X)/\text{Pic}^0(X)$ is finite where $\text{Pic}^\text{tor}(X) := \{L\in \text{Pic}(X): L\equiv 0\}$.
		\end{proof}
		\begin{corollary}[{\cite[Corollary 1.2]{BBT23}}]
			All algebraic spaces $U_i$ are quasi-projective.
		\end{corollary}
		\begin{corollary}
			Keep the notation and assumptions of \cref{const: main}. Let $P$ be a connected component of $(\Phi^\CY)^{-1}(\star)$ for some point $\star\in \Gamma^\CY\backslash\D^\CY$. Then $(\X_s, \B_{\X_s}) \approx (\X_t, \B_{\X_t})$ for all (closed) points $s,t\in P$.
		\end{corollary}
		
	\section{Applications}\label{sec: applications}
		\subsection{Smooth case.}
		\begin{conjecture}
			\label{conj: rigidity(smooth)}
			Let $\X/S$ be a smooth projective family of varieties such that $K_{\X_s}$ is nef for every closed point $s\in S$. Suppose that $\X/S$ is potentially birationally isotrivial. Then $\X/S$ is isotrivial.
		\end{conjecture}
			Classically, conjecture holds for one extremal case when the fibers have maximal Kodaira dimension, i. e. $\kappa(\X_s) = \dim \X/S$ for all $s\in S$. So, we prove it for another extremal case: $\kappa(\X_s) = 0$. Although \cref{th: torsion_iso} is a particular case of \cref{th: 0-pair_iso}, its proof is rather elementary and illustrative.
		\begin{theorem}
			\label{th: torsion_iso}
			Let $\X/S$ be a smooth projective family of K-torsion varieties. Suppose that $\X/S$ is potentially birationally isotrivial. Then $\X/S$ is isotrivial.
		\end{theorem}
		\begin{proof}
			Fix a closed point $o\in S$, and consider its \'etale neighborhood $U$ as in \cref{lem: BB}. Then there is an \'etale Galois cover $\X'\to \X_U$ such that $\X' = \prod_i \X_i'$, and $\Omega_{\X'_s}^d\approx \Os_{\X'_s}$ for all $s\in U$, here $d = \dim \X'/U$. By the hypothesis, there is variety $X$ such that $\text{Bir}_X(\X/S)$ is dense in $S$. Put $X' = \X'_{s_0}$ for some closed point $s_0\in \text{Bir}_X(\X_U/U)$. Then we can assume that $\text{Bir}_{X'}(\X'/U)$ is also dense in $U$.  Let $\Phi^\tr: U^\an \to \Gamma^\tr \backslash \D^\tr$ be a period map as in \cref{constr: period_maps_smooth}. Since $H^{d, 0}(X') = H^0(X', \Omega^d_{X'})$ is a birational invariant, we obtain isomorphism  $\V^\tr_s\approx \V^\tr_t$ of unpolarized Hodge structures for all closed points $s,t \in \text{Bir}_{X'}(\X'/U)$. By finiteness of polarizations for a fixed Hodge structure \cite[Lemma 2.22]{BFMT25}, the image $\text{Bir}_{X'}(\X'/U)$ under $\Phi^\tr$ is finite. Meanwhile, $\text{Bir}_{X'}(\X'/U)$ is dense in $U$, and $U$ is irreducible. This implies that $\Phi (U^\an)$ is a one-point set. By \cref{th: main_smooth}, $\X_U/U$ is isotrivial. This concludes the proof because the argument applies to every closed point $o\in S$. 
		\end{proof}
		
		\subsection{General case.}
		\begin{conjecture}
			\label{conj: rigidity}
			Let $(\X/S, \B)$ be a projective elementary klt family such that every fiber is a wlc model.  Suppose that $(\X/S, \B)$ is potentially birationally isotrivial. Then $(\X/S, \B)$ is isotrivial.
		\end{conjecture}
		\begin{remark}
			The conjecture holds when every fiber is of (log) general type.
		\end{remark}
		
		\begin{theorem}
			\label{th: 0-pair_iso}
			Let $(\X/S, \B)$ be a projective elementary klt family which is a lc-trivial fibration.  Suppose that $(\X/S, \B)$ is potentially birationally isotrivial, and $\B$ is effective. Then $(\X/S, \B)$ is isotrivial.
		\end{theorem}
		\begin{proof}
			By the hypothesis, there is a proper log pair $(X, B)$ such that $\text{Bir}_{(X, B)}(\X/S, \B)$ is dense in $S$. From definition of a projective elementary klt family it follows that $(\X/S, \B)$ is a locally stable family. Choose an ample invertible sheaf $\A$ on $\X/S$. Then there is a morphism $\Phi: S\to M_{d, N}$ where $d = \dim X, N = \A^d_{s_0}$ for any point $s_0\in S$. Then the image $\Phi(\text{Bir}_{(X, B)}(\X/S, \B))$ is finite by \cref{th: moduli_finite}. Meanwhile, $\text{Bir}_{(X, B)}(\X/S, \B)$ is dense in $S$, and $S$ is irreducible. This implies that $\Phi(S)$ is a one-point set, and the theorem follows.
		\end{proof}
		
		We conclude the paper by providing an example showing that \cref{th: 0-pair_iso} cannot be strengthened to the lc case. The main obstacle is that the set $\Iso_{(X,B)}(\X/S, \B)$ may not be constructible. 
		\begin{example}
			Let $C\subset \Pp^2_\Ko$ be a nodal cubic. Then $(\Pp^2_\Ko, C)$ is a projective lc $0$-pair. Choose $9$ very general regular points $p_0, p_1, \dots, p_8$ on $C$. Then the rational surface $X=\text{Bl}_{p_0, p_1, \dots p_8} \Pp^2_\Ko$ includes no $(-2)$-curves. Let $B$ be the strict transform of $C$ to $X$. Then $(X, B)$ is a \textit{generic} Loijenga pair \cite[Definition 1.4]{GHK15}. By \cite[Example 5.11]{Ser26}, there is a projective elementary family $(\X/S, \B)$ such that $\Iso_{(X, B)}(\X/S, \B)$ is dense but not constructible. Meanwhile, $\text{Bir}_{(X, B)}(\X/S, \B) = S$.
		\end{example}

\end{document}